\documentclass[12pt]{article}
\usepackage{amsmath,amssymb,amsfonts,amsthm,indentfirst,setspace}
\usepackage{latexsym,bm}
\usepackage{mathrsfs,amsmath,amssymb,amsthm}
\usepackage[dvipdfm]{graphicx}

\numberwithin{equation}{section}

\theoremstyle{plain}
\newtheorem{theorem}{Theorem}[section]

\newtheorem{lemma}{Lemma}[section]

\newtheorem{corollary}{Corollary}[section]
\newtheorem{remark}{Remark}[section]

\newcommand{\be}{\begin{equation}}
\newcommand{\ee}{\end{equation}}
\newcommand{\bea}{\begin{eqnarray}}
\newcommand{\eea}{\end{eqnarray}}
\newcommand{\eeas}{\end{eqnarray*}}
\newcommand{\beas}{\begin{eqnarray*}}
\usepackage{amscd}

\numberwithin{equation}{section}

\makeatletter
\newcommand{\trace}{\mathop{\operator@font Trace}}
\newcommand{\vspan}{\mathop{\operator@font Span}}
\newcommand{\Int}{\mathop{\operator@font Int}}
\newcommand{\grad}{\mathop{\operator@font grad}}
\newcommand{\diver}{\mathop{\operator@font div}}
\newcommand{\Ad}{\mathop{\operator@font Ad}}
\newcommand{\id}{\mathop{\operator@font id}}

\makeatother

\begin{document}

\title{Hopf hypersurfaces in complex Grassmannians of rank two
}

\author{Ruenn-Huah Lee and Tee-How {Loo}\\
Department of Financial Mathematics and Statistics,\\
Sunway University Business School,\\
No. 2, Jalan Universiti, Bandar Sunway,\\
47500 Selangor, Malaysia.\\
Institute of Mathematical Sciences, University of Malaya \\
50603 Kuala Lumpur, Malaysia.	\\ 
\ttfamily{rhlee063@hotmail.com }, 
\ttfamily{looth@um.edu.my}}

\date{}
\maketitle

\abstract{
In this paper, we study real hypersurfaces in complex Grassmannians of rank two.
First, the nonexistence of mixed foliate real hypersurfaces is proven. With this result, 
we show that for Hopf hypersurfaces in complex Grassmannians of rank two, the Reeb principal curvature is constant along integral curves of the Reeb vector field. As a result the classification of contact real hypersurfaces is obtained.
We also introduce the notion of  $q$-umbilical real hypersurfaces in complex Grassmannians of rank two and obtain a classification of such real hypersurfaces.
}




\section{Introduction}

The complex Grassmannians of rank two (both the compact type: $SU_{m+2}/S(U_2U_m)$ and the noncompact type: $SU_{2,m}/S(U_2U_m)$)
of complex dimension $2m$ are Riemannian symmetric spaces equipped with a K\"{a}hler structure $J$ and a quaternionic K\"{a}hler structure $\mathfrak J$.  Another interesting characteristic is the presence of the real structure $JJ_a$, $a\in\{1,2,3\}$,  on its tangent spaces, arisen from the interaction between $J$ and $\mathfrak J$, where $\{J_1,J_2,J_3\}$ is a canonical local basis  for $\mathfrak J$.

These three geometric structures  significantly impose restrictions on the geometry of a real hypersurface $M$ in complex Grassmannians of rank two.  As an immediate consequence of the Codazzi equation of such submanifolds, the totally umbilicity are too strong to be satisfied by real hypersurfaces in complex Grassmannians of rank two.

Apart from the submanifold structure, three additional structures are  then induced on $M$ by these geometric structures of the ambient spaces: an almost contact structure $(\phi,\xi,\eta)$ on $M$ from the K\"ahler structure $J$; 
an almost contact 3-structure $(\phi_a,\xi_a,\eta_a)$, $a\in\{1,2,3\}$  from the quaternionic K\"ahler structure $\mathfrak J$; and 
local endomorphisms $\theta_a:=\phi_a\phi-\xi_a\otimes\eta$ on $TM$, $a\in\{1,2,3\}$, from the interaction bewteen $J$ and $\mathfrak J$. 

The formulations of the induced almost contact structure and almost contact $3$-structure of real hypersurfaces $M$ were well established and have been widely used in studying the geometry of real hypersurfaces in the literature.
In contrast,  less is known about  the characteristics of the local endomorphism $\theta_a$. 
In this paper, we establish a complete algebraic formulation for $\theta_a$. With this notion, we introduce a concept, which is
so-called $q$-umbilicity.
To some extent, $q$-umbilical real hypersurfaces are those in complex Grassmannians of rank two with the richest  geometric characteristics due to the nonexistence of totally umbilical real hypersurfaces. 

A real hypersurface $M$ in complex Grassmannians of rank two is said to be \emph{$q$-umbilical} if  the shape operator $A$ of $M$ satisfies
\begin{align*}
A= f_1\mathbb I+f_2\theta +f_3\sum^3_{a=1}\xi_a\otimes\eta_a
\end{align*}
where $f_1$, $f_2$, $f_3$  are functions on $M$ and $\theta:=\sum^3_{a=1}\eta_a(\xi)\theta_a$.

The concept of $q$-umbilicity was formulated in such a way after having taken into account the restrictions on $M$  imposed by the three geometric structures of the ambient spaces. 
The absence of an  almost contact structure on $M$  under this condition is justified in the last section.

This paper is organized as follows:
After a quick revision on the geometric structures on complex Grassmannians of rank two in Sect. 2 and 
some well-known structural equations on its real hypersurface $M$ in the first half of Sect. 3, 
we establish some fundamental equations regarding the local endomorphism $\theta_a$ in the second half of Sect. 3.
We also introduce an endomorphism $\theta$ on $TM$ and obtain some of its properties in Sect. 3.
In Sect. 4, we focus on Hopf hypersurfaces in  complex Grassmannians of rank two. Some characteristics of the principal curvatures and the corresponding principal curvature spaces for Hopf hypersurfaces are studied. 
The nonexistence of mixed foliate real hypersurfaces in  complex Grassmannians of rank two is obtained in Sect. 5.
As an application of this result, 
we show that for Hopf hypersurfaces in complex Grassmannians of rank two, the Reeb principal curvature is constant along integral curves of the Reeb vector field. As a result, we can complete the classification of contact real hypersurfaces in $SU_{2,m}/S(U_2U_m)$ initiated in \cite{berndt-lee-suh}. In the last section, we classify  $q$-umbilical real hypersurfaces in complex Grassmannians of rank two.

\section{The complex Grassmannians of rank two}
We recall some geometric structures on complex Grassmannian of rank two in this section 
(see \cite{berndt}--\cite{berndt-suh2} for details).

The complex Grassmannian  $SU_{m+2}/S(U_2U_m)$ of all complex two-dimensional linear subspaces in $\mathbb C^{m+2}$ is a connected, simply connected irreducible Riemannian symmetric space of compact type and with rank two.
Let  $G=SU_{m+2}$ and $K=S(U_2U_m)$. Denote by $\mathfrak g$ and $\mathfrak k$ the corresponding Lie algebra.
Let $\mathfrak m=\mathfrak k^\perp$ with respect to the Killing form $B$ of $\mathfrak g$.
Then $\mathfrak m$ is $\Ad(K)$-invariant and we obtain a reductive decomposition $\mathfrak g=\mathfrak k\oplus\mathfrak m$.
The negative of $B$ defines a positive definite inner product on $\mathfrak m$. Denote by $g$ the corresponding $G$-invariant Riemannian metric on $SU_{m+2}/S(U_2U_m)$, 
we rescale  $g$ such that the maximal sectional curvature of $SU_{m+2}/S(U_2U_m)$ is $8c$, where $c>0$ is a constant.

The Lie algebra $\mathfrak k$ decomposes orthogonally into $\mathfrak k=\mathfrak {su}_2\oplus\mathfrak {su}_m\oplus\Re$, where 
$\Re$ is the center of $\mathfrak k$. Viewing $\mathfrak k$ as the holonomy algebra $SU_{m+2}/S(U_2U_m)$, the center $\Re$ induces a 
K\"ahler structure $J$, and $\mathfrak {su}_2$ induces a quaternionic K\"ahler structure $\mathfrak J$  on $SU_{m+2}/S(U_2U_m)$.

The complex Grassmannian  $SU_{2,m}/S(U_2U_m)$ of all positive definite complex two-dimensional linear subspaces in $\mathbb C^{m+2}_2$
 is a connected, simply connected irreducible Riemannian symmetric space of noncompact type and with rank two.
Let  $G=SU_{2,m}$ and $K=S(U_2U_m)$. Denote by $\mathfrak g$ and $\mathfrak k$ the corresponding Lie algebra.
Consider the Cartan involution $\sigma$ on $\mathfrak g$ given by 
$\sigma(A)=SAS^{-1}$, where $S=\left[\begin{matrix}-I_2 & 0 \\ 0 & I_m\end{matrix}\right]$.
Then $B_\sigma(X,Y):=- B(X,\sigma Y)$ is a positive definite $\Ad(K)$-invariant inner product on $\mathfrak g$, 
where $B$ is the Killing form of $G$. Let $(\mathfrak k,\mathfrak m)$ be a Cartan pair of $\mathfrak g$ associated to the Cartan involution $\sigma$.
Then the restriction of $B_\sigma$ to $\mathfrak m$ induces a Riemannian metric $g$ on $SU_{2,m}/S(U_2U_m)$, which is unique up to a scaling. 
We take a scaling factor $c<0$  such that the minimal sectional curvature of $SU_{2,m}/S(U_2U_m)$ is $8c$.

The Lie algebra $\mathfrak k$ can be decomposed orthogonally as  $\mathfrak k=\mathfrak {su}_2\oplus\mathfrak {su}_m\oplus\mathfrak u_1$, where $\mathfrak u_1$ is the center of $\mathfrak k$.
The adjoint action of $\mathfrak {su}_2$ on ${\mathfrak m}$ induces a  quaternionic K\"{a}hler structure ${\mathfrak J}$,  
and the adjoint action of \begin{equation*}
Z = \left(\begin{array}{cc}
 \frac{mi}{m+2}I_2 & 0 \\
 0 & \frac{-2i}{m+2}I_m
 \end{array}\right)
 \in {\mathfrak u}_1
\end{equation*}
induces a K\"{a}hler structure $J$ on $SU_{2,m}/S(U_2U_m)$ respectively.

In this paper, we use a unified notation. Denote by  $\hat M^m(c)$ the compact complex Grassmannian $ SU_{m+2}/S(U_2U_m)$ of rank two 
 (resp.  noncompact complex Grassmannian $SU_{2,m}/S(U_2U_m)$ of  rank two) for $c>0$ (resp. $c<0$), where    $c$ is a scaling factor for the Riemannian metric $g$.

For each $x\in \hat M^m(c)$,  let  $\{J_1,J_2,J_3\}$ be a canonical local basis of $\mathfrak J$
on a neighborhood $\mathcal U$ of $x\in \hat M(c)$, that is, each $J_a$ is a local almost Hermitian structure 
such that
\begin{align}\label{eqn:quaternion}
J_aJ_{a+1}=J_{a+2}=-J_{a+1}J_a, \quad a\in\{1,2,3\}.  
\end{align}
Here, the index $a$ is taken modulo three.
Denote by $\hat\nabla$ the Levi-Civita connection of $\hat M^m(c)$.
There exists local $1$-forms $q_1$, $q_2$ and $q_3$ such that
\[
\hat\nabla_XJ_a=q_{a+2}(X)J_{a+1}-q_{a+1}(X)J_{a+2}
\]
for any $X\in T_x\hat M^m(c)$, that is, $\mathfrak J$ is parallel with respect to $\hat\nabla$.
The K\"ahler structure $J$ and quarternionic K\"ahler structure $\mathfrak J$ are related by 
\begin{align}\label{eqn:JJa}
JJ_a=J_aJ; \quad \trace{(JJ_a)}=0, \quad a\in\{1,2,3\}.
\end{align}

The Riemannian curvature tensor $\hat R$ of $\hat M(c)$ is locally given by
\begin{align}\label{eqn:hatR}
\hat R(X,Y)Z=&c\{g( Y,Z) X-g( X,Z) Y+g( JY,Z) JX				\nonumber\\
													&-g( JX,Z) JY-2g( JX,Y) JZ	\}\nonumber\\
&+c\sum_{a=1}^3\{g( J_aY,Z) J_aX-g( J_aX,Z) J_aY-2g( J_aX,Y) J_aZ							\nonumber\\
&+g( JJ_aY,Z) JJ_aX-g( JJ_aX,Z) JJ_aY\}.
\end{align}
for all $X$, $Y$ and $Z\in T_x\hat M^m(c)$.

For a nonzero vector $X\in T_x\hat M^m(c)$, we denote by 
$\mathfrak JX=\{J'X|J'\in\mathfrak J_x\}$.
Recall that a maximal flat in a $\hat {M}^m(c)$ is a connected complete,  totally geodesic flat submanifold
of maximal dimension. 
Let  $X\in T\hat M^m(c)$ be a non-zero vector. Then $X$ is said to be \emph{singular} 
if  it is contained in more than one maximal flat in $\hat M^m(c)$.
It is well-known that $X$ is singular if and only if either $JX \in \mathfrak J X$ or $JX \perp \mathfrak J X$.

\section{Real hypersurfaces in $\hat M^m(c)$}

In this section, we prepare and derive some fundamental identities for real hypersurfaces in $\hat M^m(c)$.
Some of these identities have been proven in \cite{ berndt-lee-suh, berndt-suh, berndt-suh2, looth, suh1}. 
Some well-known results are also stated.

Let $M$ be a connected, oriented real hypersurface isometrically immersed in $\hat M^m(c)$, $m\geq 3$, 
and $N$ a unit normal vector field on $M$. Denote by the same $g$ the Riemannian metric on $M$. 
A canonical local basis $\{J_1,J_2,J_3\}$ of $\mathfrak J$ on $\hat M^m(c)$ induces an almost contact metric 3-structure $(\phi_a,\xi_a,\eta_a,g)$ on $M$ by
$$
J_a X=\phi_a X+\eta_a (X) N,\quad\quad J_a N=-\xi_a,\quad\quad \eta_a(X)=g(X,\xi_a),
$$
for any $X\in TM$. It follows that 
\begin{align}\label{eqn:3-contact}
\left.\begin{aligned}
\phi_a\phi_{a+1}-\xi_a\otimes\eta_{a+1}&=\phi_{a+2}=-\phi_{a+1}\phi_a+\xi_{a+1}\otimes\eta_a \\ 
\phi_a\xi_{a+1}=\xi_{a+2}&=-\phi_{a+1}\xi_a  
\end{aligned}\right\}
\end{align}
for $a\in\{1,2,3\}$. The indices in the preceding equations are taken modulo three.

Let $(\phi, \xi,\eta,g)$ be the almost contact metric structure on $M$ induced by $J$,  that is,  
\begin{align*}	
JX=\phi X+\eta(X)N,	\quad JN=-\xi, \quad \eta(X)=g(X,\xi). 
\end{align*}
The two structures $(\phi,\xi,\eta,g)$ and $(\phi_a,\xi_a,\eta_a,g)$
are related as follows  
\begin{align}\label{eqn:1}
\phi_a\phi-\xi_a\otimes\eta=\phi\phi_a-\xi\otimes\eta_a; \quad \phi\xi_a=\phi_a\xi.
\end{align}

Next, we denote by $\nabla$ the Levi-Civita connection and $A$ the shape operator on $M$. Then 
\begin{eqnarray}\label{eqn:contact}
\left.\begin{aligned}
&(\nabla_{X} \phi)Y=\eta(Y)AX-g( AX,Y)\xi, \quad \nabla_X \xi = \phi AX \\
&(\nabla_{X}\phi_a)Y=\eta_a (Y)AX- g(AX,Y)\xi_{a}
                                        +q_{a+2}(X)\phi_{a+1}Y-q_{a+1}(X)\phi_{a+2}Y \\
&\nabla_X \xi_a = \phi_a AX+q_{a+2}(X)\xi_{a+1}-q_{a+1}(X)\xi_{a+2} \\
&X\eta(\xi_a)=2\eta_a(\phi AX)+\eta_{a+1}(\xi)q_{a+2}(X)- \eta_{a+2}(\xi)q_{a+1}(X).
\end{aligned}\right\}
\end{eqnarray}

Let $\mathfrak D^\perp=\mathfrak JN$, and $\mathfrak D$ its orthogonal complement in $TM$. 
If $\xi\in\mathfrak D$, then $\eta(\xi_a)=0$ for $a\in\{1,2,3\}$ and so by the preceding equation, we obtain 
\begin{lemma}\label{lem:Aphixi_a}
If $\xi\in\mathfrak D$, then $A\phi\xi_a=0$ for $a\in\{1,2,3\}$.
\end{lemma}

We define a local  symmetric $(1,1)$-tensor field $\theta_a$ on $M$  by
\[
\theta_a :=\phi_a\phi -\xi_a\otimes\eta.
\]
Then we have the following identities 
\begin{lemma}\label{lem:theta}
\begin{enumerate}
\item[(a)] $\theta_a$ is symmetric,
\item[(b)] $\trace{(\theta_a)}=\eta(\xi_a)$,
\item[(c)] $\theta_a\xi=-\xi_a; \quad \theta_a\xi_a=-\xi; \quad \theta_a\phi\xi_a=\eta(\xi_a)\phi\xi_a$, 
\item[(d)]  $\theta_a\xi_{a+1}= \phi\xi_{a+2}=-\theta_{a+1}\xi_a$,
\item[(e)] $\theta_a^2-\phi\xi_a\otimes\eta_a\phi=\mathbb I$,
\item[(f)] $-\theta_a\phi\xi_{a+1}+\eta(\xi_{a+1})\phi\xi_a =\xi_{a+2}=\theta_{a+1}\phi\xi_a-\eta(\xi_a)\phi\xi_{a+1}$,
\item[(g)] $-\theta_a\theta_{a+1}+\phi\xi_a\otimes\eta_{a+1}\phi=\phi_{a+2}
											=\theta_{a+1}\theta_a-\phi\xi_{a+1}\otimes\eta_a\phi$,
\item[(h)] $\theta_a\phi-\phi\xi_a\otimes\eta=-\phi_a=\phi\theta_a-\xi\otimes\eta_a\phi$,				
\item[(i)] $\theta_a\phi_a-\phi\xi_a\otimes\eta_a=-\phi=\phi_a\theta_a-\xi_a\otimes\eta_a\phi$,											
\item[(j)] $\theta_a\phi_{a+1}-\phi\xi_a\otimes\eta_{a+1}=\theta_{a+2}=-\phi_{a+1}\theta_a-\xi_{a+1}\otimes\eta_a\phi$.													
\end{enumerate}
\end{lemma}
\begin{proof}
(a)--(f)  The proof is exactly the same as that given in  \cite{looth}.

\medskip
(g)  By  using (\ref{eqn:3-contact})--(\ref{eqn:1}), we have
\begin{align*}
\theta_a\theta_{a+1}-\phi\xi_a\otimes\eta_{a+1}\phi
=&(\phi_a\phi-\xi_a\otimes\eta)(\phi\phi_{a+1}-\xi\otimes \eta_{a+1})-\phi\xi_a\otimes\eta_{a+1}\phi \nonumber\\
=&\phi_a\phi^2\phi_{a+1}+\xi_a\otimes\eta_{a+1}-\phi\xi_a\otimes\eta_{a+1}\phi\nonumber\\
=&\phi_a(-\mathbb I+\xi\otimes\eta)\phi_{a+1}+\xi_a\otimes\eta_{a+1}-\phi\xi_a\otimes\eta_{a+1}\phi\nonumber\\
=&-\phi_a\phi_{a+1}+\xi_a\otimes\eta_{a+1}\nonumber\\
=&-\phi_{a+2}.
\end{align*}
The second equality can be obtained as follows
\begin{align*}
\phi_{a+2}=(-\phi_{a+2})^*=&(\theta_a\theta_{a+1}-\phi\xi_a\otimes\eta_{a+1}\phi)^*
=\theta_{a+1}\theta_a-\phi\xi_{a+1}\otimes\eta_a\phi
\end{align*}
where  we denote by $T^*$ the adjoint of  an endomorphism $T$ with respect to $g$. 

\medskip
(h)--(j) The first equalities can be obtained as follows
\begin{align*}
\theta_a\phi&=(\phi_a\phi-\xi_a\otimes\eta)\phi =\phi_a(-\mathbb I+\xi\otimes\eta)=-\phi_a+\phi\xi_a\otimes\eta			\\
\theta_a\phi_a&=(\phi\phi_a-\xi\otimes\eta_a)\phi_a =\phi(-\mathbb I+\xi_a\otimes\eta_a)=-\phi+\phi\xi_a\otimes\eta_a \\										
\theta_a\phi_{a+1}&=(\phi\phi_a-\xi\otimes\eta_a)\phi_{a+1} =\phi(\phi_{a+2}+\xi_a\otimes\eta_{a+1})-\xi\otimes\eta_{a+2} \\		
&=(\phi\phi_{a+2}-\xi\otimes\eta_{a+2})+\phi\xi_a\otimes\eta_{a+1}=\theta_{a+2}+\phi\xi_a\otimes\eta_{a+1}.
\end{align*}
In  a similar manner as in (g), we can obtain the second equalities for these parts.
\end{proof}

Note that 
\begin{align*}
(\nabla_X\theta_a)Y
=&(\nabla_X\phi)\phi_aY+\phi(\nabla_X\phi_a)Y-g(\nabla_X\xi_a,Y)\xi-\eta_a(Y)\nabla_X\xi \\
\nabla_X\phi\xi_a
=&(\nabla_X\phi)\xi_a+\phi\nabla_X\xi_a. 
\end{align*}
Then by applying  (\ref{eqn:contact}), we obtain
\begin{align}\label{eqn:del_theta-a}
\left.\begin{aligned}
(\nabla_X\theta_a)Y
=&\eta_a(\phi Y)AX-g(AX,Y)\phi\xi_a+q_{a+2}(X)\theta_{a+1}Y-q_{a+1}(X)\theta_{a+2}Y \\
\nabla_X\phi\xi_a
=&\theta_aAX+\eta_a(\xi)AX+q_{a+2}(X)\phi\xi_{a+1} - q_{a+1}(X)\phi\xi_{a+2}. 
\end{aligned}\right\}
\end{align}

For each $x\in M$,   we define a subspace $\mathcal H^\perp$ of $T_xM$ by
$$\mathcal H^\perp: =\mathrm{span}\{\xi,\xi_1,\xi_2,\xi_3,\phi\xi_1,\phi\xi_2,\phi\xi_3\}.$$
Let $\mathcal{H}$ be the orthogonal complement of $\mathcal {H}^\perp$ in $T_xM$. 
Then  
$\dim\mathcal H=4m-4$  (resp. $\dim\mathcal H=4m-8$) when $\xi\in\mathfrak D^\perp$ (resp. $\xi\notin\mathfrak D^\perp$)
and 
$\mathcal{H}$ is invariant under $\phi,\phi_a$ and $\theta_a$. Moreover, $\theta_{a|\mathcal{H}}$ has two eigenvalues: $1$ and $-1$. 
Denote by  $\mathcal H_a(\varepsilon)$ the eigenspace corresponding to the eigenvalue $\varepsilon$ of 
${\theta_a}_{|\mathcal H}$.
Then $\dim \mathcal H_a(1)=\dim \mathcal H_a(-1)$ is even, and 
\begin{align}
\left.\begin{aligned}\label{eqn:Ha}
&\phi\mathcal H_a(\varepsilon)=\phi_a\mathcal H_a(\varepsilon)=\theta_a\mathcal H_a(\varepsilon)=\mathcal H_a(\varepsilon) \\
&\phi_b\mathcal H_a(\varepsilon)=\theta_b\mathcal H_a(\varepsilon)=\mathcal H_a(-\varepsilon), \quad (a\neq b).
\end{aligned}\right\}
\end{align}
The proof of (\ref{eqn:Ha}) is exactly the same as presented in \cite[pp. 92--93]{looth}.

Observe that $\tan(JJ_aX)=\theta_aX$ and $\mathrm{nor}(JJ_aX)=\eta_a(\phi X)N$, for $X\in TM$.
Then the equations of Gauss and Codazzi are respectively given by
$$\begin{aligned}
R(X,Y)Z=&g( AY,Z) AX-g( AX,Z) AY+c\{g( Y,Z) X-g(X,Z) Y\\
&+g(\phi Y,Z)\phi X-g(\phi X,Z)\phi Y -2g(\phi X,Y)\phi Z\}\\
&+c\sum_{a=1}^3\{g(\phi_aY,Z)\phi_aX-g(\phi_aX,Z) \phi_aY-2g(\phi_aX,Y)\phi_aZ\\
&+g(\theta_aY,Z)\theta_aX-g(\theta_aX,Z)\theta_aY\}
\end{aligned}$$
\begin{align*}
(\nabla_X A)Y-(\nabla_Y A)X=&c\{\eta(X)\phi Y-\eta(Y)\phi X-2g(\phi X,Y)\xi\}\\
&+c\sum_{a=1}^3 \{\eta_a(X)\phi_a Y-\eta_a(Y)\phi_a X -2g(\phi_a X,Y)\xi_a\\&
+\eta_a(\phi X)\theta_a Y-\eta_a(\phi Y)\theta_a X\}.
\end{align*}

We define the tensor fields $\theta$, $\phi^\perp$, $\xi^\perp$ and $\eta^\perp$  on $M$ as follows
\begin{align*}
\theta:=&\sum^3_{a=1}\eta_a(\xi)\theta_a, \quad \phi^\perp:=\sum^3_{a=1}\eta_a(\xi)\phi_a, \quad 
\xi^\perp:=\sum^3_{a=1}\eta_a(\xi)\xi_a, \quad \eta^\perp:=\sum^3_{a=1}\eta_a(\xi)\eta_a.
\end{align*}

\begin{lemma}\label{lem:H-global}
At each $x\in M$ with $\xi^\perp\neq0$,  $\theta_{|\mathcal H}$ has two eigenvalues $\varepsilon||\xi^\perp||$, $\varepsilon\in\{1,-1\}$.
Let $\mathcal H(\varepsilon)$ be the eigenspace of $\theta_{|\mathcal H}$ corresponding to  $\varepsilon||\xi^\perp||$. Then 
\begin{enumerate}
\item[(a)] $\phi \mathcal H(\varepsilon)=\phi^\perp \mathcal H(\varepsilon)= \mathcal H(\varepsilon)$,
\item[(b)] $\dim \mathcal H(1)=\dim \mathcal H(-1)$ is even.
\end{enumerate}
\end{lemma}
\begin{proof}
We take a canonical local basis of $\mathfrak J$ such that 
$\xi_1=\xi^\perp/||\xi^\perp||$, so $\theta=\eta_1(\xi)\theta_1$, $\phi^\perp=\eta_1(\xi)\phi_1$ and $\eta^\perp=\eta_1(\xi)\eta_1$. 
It follows that $\mathcal H(\varepsilon)=\mathcal H_1(\varepsilon)$ and so  these results  follow from (\ref{eqn:Ha}).
\end{proof}

As in the proof of Lemma~\ref{lem:H-global}, at each $x\in M$ with $||\xi^\perp||\neq0$, we can take a canonical local basis of $\mathfrak J$ such that on a neighborhood $G\subset M$ of $x$, we have
\begin{align}\label{eqn:natural-basis}
\left.\begin{aligned}
&\xi_1=\frac{\xi^\perp}{||\xi^\perp||}, \quad  0<\eta_1(\xi)=||\xi^\perp||\leq 1, \quad \eta_2(\xi)=\eta_3(\xi)=0\\
&\mathcal H(\varepsilon)=\mathcal H_1(\varepsilon),  \quad \theta=\eta_1(\xi)\theta_1,   \quad \phi^\perp=\eta_1(\xi)\phi_1,
				\quad \eta^\perp=\eta_1(\xi)\eta_1
\end{aligned}\right\}
\end{align}
where  $\mathcal H(\varepsilon)$  is the eigenspace of $\theta_{|\mathcal H}$ corresponding to  an eigenvalue $\varepsilon||\xi^\perp||$ of 
$\theta_{|\mathcal H}$ for $\varepsilon\in\{1.-1\}$.  
Furthermore if $||\xi^\perp||=1$ at  $x$, then
\begin{align}\label{eqn:natural-basis2}
&\xi_1=\xi=\xi^\perp, \quad  \xi_2=\theta\xi_2=\phi\xi_3, \quad \xi_3=\theta\xi_3=-\phi\xi_2
\end{align}

Throughout this paper, we always consider such a local orthonormal frame $\{\xi_1,\xi_2,\xi_3\}$ on $\mathfrak D^\perp$ under these situations.

\begin{lemma}\label{lem:global}
\begin{enumerate}
\item[(a)] $\theta\xi^\perp=-||\xi^\perp||^2\xi$, \quad   $\theta\xi=-\xi^\perp,  \quad \theta\phi\xi^\perp=||\xi^\perp||^2\phi\xi^\perp$, 

\item[(b)]	
									$\theta^2-\phi\xi^\perp\otimes\eta^\perp\phi=||\xi^\perp||^2\mathbb I$,  
\item [(c)] $\theta\phi-\phi\xi^\perp\otimes\eta=-\phi^\perp=\phi\theta-\xi\otimes\eta^\perp\phi $,
\item[(d)] $(\phi^\perp)^2 =-||\xi^\perp||^2\mathbb I+\xi^\perp\otimes\eta^\perp$,
\item[(e)] $\phi^\perp\phi-\xi^\perp\otimes\eta=\theta=\phi\phi^\perp-\xi\otimes\eta^\perp $,
\item[(f)]  $d(||\xi^\perp||^2)=4\eta^\perp\phi A $,
\item[(g)] $(\nabla_X\phi^\perp)Y
																=\eta^\perp(Y)AX-g(AX,Y)\xi^\perp+2\sum^3_{a=1}\eta_a(\phi AX)\phi_aY$
\item[(h)] $\nabla_X\xi^\perp=\phi^\perp AX+2\sum^3_{a=1}\eta_a(\phi AX)\xi_a$
\item[(i)] $(\nabla_X\theta)Y=\eta^\perp(\phi Y)AX-g(AX,Y)\phi\xi^\perp
															+2\sum^3_{a=1}\eta_a(\phi AX)\theta_aY $
\item[(j)] 	$\nabla_X\phi\xi^\perp=\theta AX+||\xi^\perp||^2AX+2\sum^3_{a=1}\eta_a(\phi AX)\phi\xi_a$.
\end{enumerate}
\end{lemma}
\begin{proof}
(a)--(e) At each $x\in M$ with $\xi^\perp=0$,  we have $\theta=\phi^\perp=0$ and $\eta^\perp=0$. Hence these identities are trivial. 
Suppose $||\xi^\perp||\neq0$ at a point $x\in M$. Then these identities can be easily obtained from Lemma~\ref{lem:theta} and 
(\ref{eqn:natural-basis}).

\medskip
(f)--(j) We only give the proof for (i) as the remaining parts can be obtained by a similar straightforward calculation.
By using (\ref{eqn:contact})--(\ref{eqn:del_theta-a}), 
\begin{align*}
(\nabla_X\theta)Y
=&\sum^3_{a=1}\{(X(\eta_a(\xi))\theta_a Y+\eta_a(\xi)(\nabla_X\theta_a)Y\}\\
=&\sum^3_{a=1}\{2\eta_a(\phi AX)\theta_aY+\eta_a(\xi)\eta_a(\phi Y)AX-g(AX,Y)\eta_a(\xi)\phi_a\xi\}\\
=&\sum^3_{a=1}2\eta_a(\phi AX)\theta_aY+\eta^\perp(\phi Y)AX-g(AX,Y)\phi\xi^\perp.
\end{align*}
\end{proof}

We now prepare some results for later use.
By using Lemma~\ref{lem:global}, we have
\begin{align}\label{eqn:201} 
d(\eta^\perp & \phi)  (X,Y) \nonumber\\
			&=-g(\nabla_X\phi\xi^\perp,Y)+g(\nabla_Y\phi\xi^\perp,X)  \nonumber\\
			&=-g((\theta A-A\theta)X,Y)+2\sum^3_{a=1}\{\eta_a(\phi AX)\eta_a(\phi Y)-\eta_a(\phi AY)\eta_a(\phi X)\}.
\end{align}
\begin{lemma}\label{lem:f}
Suppose $0<||\xi^\perp||<1$. 
If $A\phi\xi^\perp=\omega\phi\xi^\perp$, then 
\begin{enumerate}
\item[(a)]$d\omega=-||\xi^\perp||^{-2}(1-||\xi^\perp||^2)^{-1}(\phi\xi^\perp\omega)\eta^\perp\phi$,
\item[(b)] $\omega d(\eta^\perp\phi)=0$.
\end{enumerate}
\end{lemma}
\begin{proof}
It follows from the hypothesis and Lemma~\ref{lem:global} that 
\[
d\omega\wedge\eta^\perp\phi+\omega d(\eta^\perp\phi)=\frac14d^2(||\xi^\perp||^2)=0.
\] 
This, together with (\ref{eqn:201}),  gives 
\begin{align*}
X\omega
=&-\dfrac{(\phi\xi^\perp\omega)\eta^\perp(\phi X)+(d\omega\wedge\eta^\perp\phi)(X,\phi\xi^\perp)}
															{||\xi^\perp||^{2}(1-||\xi^\perp||^2)}	
=-\dfrac{(\phi\xi^\perp\omega)\eta^\perp(\phi X)}
															{||\xi^\perp||^{2}(1-||\xi^\perp||^2)}.	
\end{align*}
This means that $d\omega=-||\xi^\perp||^{-2}(1-||\xi^\perp||^2)^{-1}(\phi\xi^\perp\omega)\eta^\perp\phi$
and hence  $\omega d(\eta^\perp\phi)=0$.
\end{proof}

\begin{lemma}\label{lem:theta_A}
Suppose $||\xi^\perp||=1$ on $M$. 
Then 
\begin{enumerate}
\item[(a)] $\theta A+ A=-2\sum^3_{a=1}\phi\xi_a\otimes\eta_a\phi A$,
\item[(b)] $A\mathcal H(1)=0$.
\end{enumerate}
 \end{lemma}
\begin{proof}
(a)  Since $\xi=\xi^\perp$, 
\begin{align*}
0=\phi\nabla_X(\xi^\perp-\xi)=(\theta A+ A)X+2\sum^3_{a=1}\eta_a(\phi AX)\phi\xi_a, \quad X\in TM.
\end{align*}

(b) Let $Y\in\mathcal H(1)$. Then  $2g(AY,X)=g(Y,(\theta A+A)X)=0$ for  $X\in TM$, which means that $A\mathcal H(1)=0$.
\end{proof}

\begin{lemma}\label{lem:general1}
At each $x\in M$ with  $||\xi^\perp||>0$,   
\[
\sum^3_{a=1}g(\theta_a X,Y)\eta_a(Z)=\sum^3_{a=1}g(\phi_a X,Y)\eta_a(Z)=0
\]
for any $X,Y,Z\in\mathcal H(\varepsilon)\oplus(\mathcal H^\perp\ominus\vspan\{\xi,\xi^\perp,\phi\xi^\perp\})$,
where $\varepsilon\in\{1,-1\}$. 

\end{lemma}
\begin{proof}
Under the setting in  (\ref{eqn:natural-basis}),  we have  
\[
\mathcal H^\perp\ominus\vspan\{\xi,\xi^\perp,\phi\xi^\perp\}=\vspan\{\xi_2,\xi_3,\phi\xi_2,\phi\xi_3\}.
\] 
It suffices to show that
\[ 
g(\theta_bX,Y)=g(\phi_bX,Y)=0, \quad 
X,Y\in\mathcal H_1(\varepsilon)\oplus\vspan\{\xi_2,\xi_3,\phi\xi_2,\phi\xi_3\}, b\in\{2,3\}.
\] 
First, we can easily verify that 
\[
\theta_b\xi_c, \phi_b\xi_c, \theta_b\phi\xi_c, \phi_b\phi\xi_c\in\vspan\{\xi,\xi_1,\phi\xi_1\}, \quad b,c\in\{2,3\}.
\]
This, together with (\ref{eqn:Ha}), gives the first equation.
\end{proof}

\begin{lemma}\label{lem:hopf_perp-1a}
Suppose  $0<||\xi^\perp||< 1$ on $M$.
If $A\mathcal H(\varepsilon)\subset\mathcal H(\varepsilon)$, where $\varepsilon\in\{1,-1\}$, then   
$\nabla_XY\perp \mathcal H^\perp\ominus\vspan\{\xi,\xi^\perp,\phi\xi^\perp\}$
for all vector fields $X,Y$ tangent to $\mathcal H(\varepsilon)$.
\end{lemma}
\begin{proof}
We adopt the basis $\{\xi_1,\xi_2,\xi_3\}$ for $\mathfrak D^\perp$ with properties (\ref{eqn:natural-basis}).
Let $X,Y$ be vector fields tangent to $\mathcal H(\varepsilon)$. 
Then for $b\in\{2,3\}$, by using (\ref{eqn:Ha})  we obtain 
\begin{align*}
g(\nabla_X\xi_b,Y)=g(\phi_bAX,Y)=0; \quad 
g(\nabla_X\phi\xi_b,Y)=g(\theta_bAX,Y)=0. \quad 
\end{align*}
Since $\mathcal H^\perp\ominus\vspan\{\xi,\xi^\perp,\phi\xi^\perp\}=\vspan\{\xi_2,\xi_3,\phi\xi_2,\phi\xi_3\}$, this gives the desired result.
\end{proof}

If $||\xi^\perp||=1$, then $\xi^\perp=\xi$ and $\mathcal H^\perp=\mathfrak D^\perp$.
Using a similar method as in the preceding proof, we obtain 
\begin{lemma}\label{lem:hopf_perp-1b}
Suppose  $||\xi^\perp||= 1$ on $M$.
If $\mathcal V\subset  \mathcal H(\varepsilon)$, where $\varepsilon\in\{1,-1\}$,  is a subbundle that is invariant under $A$, then 
$\nabla_XY\perp \mathfrak D^\perp\ominus\mathbb R\xi$
for all vector fields $X,Y$ tangent to $\mathcal V$.
\end{lemma}

At the end of this section, we state some important results for later use.
\begin{theorem}[\cite{berndt-suh}]\label{thm:+}
Let $M$ be a connected real hypersurface in $SU_{m+2}/S(U_2U_m)$, $m\geq3$. 
Then both $\mathbb R\xi$ and $\mathfrak D^\perp$ are invariant under the shape operator of $M$ if and only if
$M$ is an open part of one of the following spaces:
\begin{enumerate}
\item[$(A)$] a tube around a totally geodesic $SU_{m+1}/S(U_2U_{m-1})$ in $SU_{m+2}/S(U_2U_m)$, or
\item[$(B)$] a tube around a totally geodesic $\mathbb  HP_n=Sp_{n+1}/Sp_1Sp_n$ in $SU_{m+2}/S(U_2U_m)$,
                  where  $m=2n$ is even. 
\end{enumerate}
\end{theorem}

The principal curvatures of  real hypersurfaces of type $A$ and $B$ stated in   Theorem~\ref{thm:+} are given as follows:
\begin{theorem}[\cite{berndt-suh}]\label{thm:+2}
Given a constant $c>0$:
\begin{enumerate}
\item[(i)] 
If  $M$ is a real hypersurface of type $A$ in $SU_{m+2}/S(U_2U_m)$, then  $\xi\in\mathfrak D^\perp$  at each point of $M$, and 
$M$ has three (for  $r=\pi/2\sqrt{8c}$, whereby  $\alpha=\mu$) or four (otherwise) distinct constant principal curvatures
\[\begin{array}{ll}
\alpha=\sqrt{8c}\cot(\sqrt{8c}r), 			\quad 	& \beta=\sqrt{2c}\cot(\sqrt{2c}r),\\
\lambda=-\sqrt{2c}\tan(\sqrt{2c}r), 	\quad  & \mu=0
\end{array}\]
with some $r\in]0,\pi/\sqrt{8c}[$\:. 
The corresponding principal curvature spaces are
\begin{align*}
T_\alpha=\mathbb R\xi, \quad   
T_\beta=\mathfrak D^\perp\ominus\mathbb R\xi,	 \quad 
T_\lambda=\mathcal H(-1), \quad 
T_\mu=\mathcal H(1).
\end{align*}

\item[(ii)] 
If  $M$ is a real hypersurface of type $B$ in $SU_{m+2}/S(U_2U_m)$, then
$\xi\in\mathfrak D$ at each point of $M$, $m=2n$ is even and 
$M$ has five distinct constant principal curvatures
\begin{align*}
\begin{array}{lll}
\alpha=-2\sqrt{c}\tan(2\sqrt{c}r),\quad  & \beta=2\sqrt{c}\cot(2\sqrt{c}r),\quad & \gamma=0,\\
\lambda=\sqrt{c}\cot(\sqrt{c}r),  \quad  & \mu=-\sqrt{c}\tan(\sqrt{c}r) &
\end{array}
\end{align*}
with some $r\in]0,\pi/4\sqrt{c}[$. 
The corresponding principal curvature spaces are
\[
T_\alpha=\mathbb R\xi,\quad 
T_\beta=\mathfrak D^\perp,\quad 
T_\gamma=\mathfrak J\xi,\quad 
T_\lambda,\quad 
T_\mu,
\]
where $T_\lambda\oplus T_\mu=\mathcal H$, $\mathfrak JT_\lambda=T_\lambda$,
$\mathfrak JT_\mu=T_\mu$, $JT_\lambda=T_\mu$.
\end{enumerate}
\end{theorem}

\begin{theorem}[\cite{berndt-suh2}]\label{thm:-}
Let $M$ be a connected real hypersurface in $SU_{2,m}/S(U_2U_m)$, $m\geq2$. 
Then both $\mathbb R\xi$ and $\mathfrak D^\perp$ are invariant under the shape operator of $M$ if and only if
one of the following holds:
\begin{enumerate}
\item[$(A)$] $M$ is an open part of a tube around a totally geodesic $SU_{2,m-1}/S(U_2U_{m-1})$ in $SU_{2,m}/S(U_2U_m)$, or
\item[$(B)$] $M$ is an open part of a tube around a totally geodesic $\mathbb  HH_n=Sp_{1,n}/Sp_1Sp_n$ in $SU_{2,m}/S(U_2U_m)$,
                  where  $m=2n$ is even, or 
\item[$(C_1)$] $M$ is an open part of a horosphere in $SU_{2,m}/S(U_2U_m)$ whose center at infinity is singular and of type $JN\in  \mathfrak JN$, or
\item[$(C_2)$]	$M$ is an open part of  a horosphere in $SU_{2,m}/S(U_2U_m)$ whose center at infinity is singular and of type $JN\perp  \mathfrak JN$, 	or 							
\item[$(D)$] the normal bundle of $M$ consists of singular tangent vectors of type $JX\perp \mathfrak JX$. Moreover, $M$ has at least four distinct principal curvatures, which are given by
\[
\alpha=2\sqrt{-c}, \quad \gamma=0, \quad \lambda=\sqrt{-c}, \quad (c<0 \text{ is a constant})
\]
with corresponding principal curvature spaces
\[
T_\alpha=\mathbb R\xi\oplus\mathfrak D^\perp, \quad 
T_\gamma=\mathfrak J\xi, \quad 
T_\lambda\subset \mathcal H.
\]
If $\mu$ is another (possibly nonconstant) principal curvature function, then $T_\mu\subset\mathcal H$, $JT_\mu\subset T_\lambda$ and 
$\mathfrak JT_\mu\subset T_\lambda$. 
\end{enumerate}
\end{theorem}

The principal curvatures of  real hypersurfaces of type $A$, $B$, $C_1$ and $C_2$ stated in   Theorem~\ref{thm:-} are given as follows:
\begin{theorem}[\cite{berndt-suh2}]\label{thm:-2}
Given a constant $c<0$:
\begin{enumerate}
\item[(i)] 
If  $M$ is a real hypersurface of type $A$ in $SU_{2,m}/S(U_2U_m)$, then  $\xi\in\mathfrak D^\perp$  at each point of $M$, and 
$M$ has four  distinct constant principal curvatures
\[\begin{array}{ll}
\alpha=\sqrt{-8c}\coth(\sqrt{8c}r), 			\quad 	& \beta=\sqrt{-2c}\coth(\sqrt{-2c}r),\\
\lambda=\sqrt{-2c}\tanh(\sqrt{-2c}r), 	\quad  & \mu=0
\end{array}\]
with some $r>0$. The corresponding principal curvature spaces are
\begin{align*}
T_\alpha=\mathbb R\xi, \quad   
T_\beta=\mathfrak D^\perp\ominus\mathbb R\xi,	 \quad 
T_\lambda=\mathcal H(-1), \quad 
T_\mu=\mathcal H(1).
\end{align*}

\item[(ii)] 
If  $M$ is a real hypersurface of type $B$ in $SU_{2,m}/S(U_2U_m)$, then
$\xi\in\mathfrak D$ at each point of $M$, $m=2n$ is even and 
$M$ has four  (for  $\tanh^2\sqrt{-c} r=1/3$, whereby  $\alpha=\lambda$) or five  (otherwise) distinct constant principal curvatures
\begin{align*}\begin{array}{lll}
\alpha=2\sqrt{-c}\tanh(2\sqrt{-c}r),\quad  & \beta=2\sqrt{-c}\coth(2\sqrt{-c}r),\quad	&	\gamma=0,\\
\lambda=\sqrt{-c}\coth(\sqrt{-c}r),\quad   &	\mu=\sqrt{-c}\tanh(\sqrt{-c}r) & 
\end{array}
\end{align*}
with some $r>0$. 
The corresponding principal curvature spaces are
\[
T_\alpha=\mathbb R\xi,\quad 
T_\beta=\mathfrak D^\perp,\quad 
T_\gamma=\mathfrak J\xi,\quad 
T_\lambda,\quad 
T_\mu,
\]
where $T_\lambda\oplus T_\mu=\mathcal H$, $\mathfrak JT_\lambda=T_\lambda$,
$\mathfrak JT_\mu=T_\mu$, $JT_\lambda=T_\mu$.

\item[(iii)] 
If  $M$ is a real hypersurface of type $C_1$ in $SU_{2,m}/S(U_2U_m)$, then  $\xi\in\mathfrak D^\perp$  at each point of $M$, and 
$M$ has three  distinct constant principal curvatures
\[
\alpha=2\sqrt{-2c}, 	\quad 	\beta=\sqrt{-2c}, \quad  \mu=0
\]
with some $r>0$. The corresponding principal curvature spaces are
\begin{align*}
T_\alpha=\mathbb R\xi, \quad  T_\beta=\mathcal H(-1)\oplus(\mathfrak D^\perp\ominus\mathbb R\xi),	 \quad  T_\mu=\mathcal H(1).
\end{align*}

\item[(vi)] 
If  $M$ is a real hypersurface of type $C_2$ in $SU_{2,m}/S(U_2U_m)$, then
$\xi\in\mathfrak D$ at each point of $M$ and 
$M$ has three  distinct constant principal curvatures
\begin{align*}
\alpha=2\sqrt{-c},\quad  &	\gamma=0, \quad \lambda=\sqrt{-c} 
\end{align*}
with some $r>0$. 
The corresponding principal curvature spaces are
\[
T_\alpha=\mathbb R\xi\oplus\mathfrak D^\perp,\quad  T_\gamma=\mathfrak J\xi,\quad T_\lambda=\mathcal H.
\]
\end{enumerate}
\end{theorem}

A  real hypersurface $M$ in a K\"ahler manifold is said to be \emph{Hopf} if the Reeb vector field $\xi$ is principal.
The principal curvature $\alpha=g(A\xi,\xi)$ is called the \emph{Reeb principal curvature} for a Hopf hypesurface $M$.
\begin{theorem}[\cite{lee-suh,suh3}]\label{thm:lee-suh}
Let $M$ be a connected Hopf  hypersurface in $\hat M^m(c)$, $m\geq3$. 
Then $\xi\in \mathfrak D$ if and only if 
\begin{enumerate}
\item[(i)]  for $c>0$: $M$ is an open part of a real hypersurface of type $B$;  or
\item[(ii)] for $c<0$:  One of the cases $(B)$, $(C_2)$ and $(D)$ in Theorem~\ref{thm:-} holds.
\end{enumerate}
\end{theorem}

We state the following lemma without proof as its proof is entirely similar to that of \cite[Theorem 1.5]{lee-loo}.

\begin{lemma}\label{lem:1.5}
Let $M$ be a connected real hypersurface in $\hat M^m(c)$, $m\geq 3$. If $A\mathfrak D\subset\mathfrak D$ and $\xi\in\mathfrak D^\perp$, then $M$ is Hopf.
\end{lemma}

By Theorem~\ref{thm:-} and Lemma~\ref{lem:1.5}, we have

\begin{theorem}\label{thm:A_A0}
Let $M$ be a connected real hypersurface in $\hat M^m(c)$, $m\geq 3$. Then $A\mathfrak D\subset\mathfrak D$ and 
$\xi\in\mathfrak D^\perp$ if and only if  
\begin{enumerate}
\item[(i)]  for $c>0$: $M$ is an open part of a real hypersurface of type $A$ given in Theorem~\ref{thm:+};  or
\item[(ii)] for $c<0$: $M$ is an open part of one of real hypersurfaces of type $A$ or $C_1$ given in Theorem~\ref{thm:-}.
\end{enumerate}

\end{theorem}

\section{Hopf hypersurfaces in $\hat M^m(c)$}
In this section, we shall derive some fundamental properties for Hopf hypersurfaces in $\hat M^m(c)$.
Suppose  $M$ is a Hopf hypersurface in $\hat M^m(c)$ with $A\xi=\alpha\xi$.
Then as derived in \cite{berndt-suh, berndt-suh2, jeong}, we have
\begin{align}  
d\alpha=&(\xi\alpha)\eta-4c\eta^\perp\phi		\label{eqn:d_alpha}\\
A\phi A- \frac\alpha 2(\phi A+&A\phi)-c(\phi+\phi^\perp)   \nonumber\\
=&c\sum^3_{a=1}\{\xi_a\otimes\eta_a\phi+\phi\xi_a\otimes\eta_a\} -2c(\xi\otimes\eta^\perp\phi+\phi\xi^\perp\otimes\eta). 
					\label{eqn:AphiA}
\end{align}
It follows from (\ref{eqn:d_alpha}) that 
\begin{align}\label{eqn:0}
d(\xi\alpha)\wedge\eta+(\xi\alpha)d\eta-4cd(\eta^\perp\phi)=0.
\end{align}
This implies that  
\begin{align*}
X\xi\alpha
=&(\xi\xi\alpha)\eta(X)+(d(\xi\alpha)\wedge\eta)(X,\xi)
=(\xi\xi\alpha)\eta(X)+4c(\eta^\perp(AX)-\alpha\eta^\perp(X))
\end{align*}
or equivalently 
\begin{align}\label{eqn:d_xi_alpha}
d(\xi\alpha)=(\xi\xi\alpha)\eta+4c(\eta^\perp A-\alpha\eta^\perp).
\end{align}
Combining (\ref{eqn:0})--(\ref{eqn:d_xi_alpha}), gives
\begin{align}\label{eqn:210}
4c(\eta^\perp A-\alpha\eta^\perp)\wedge\eta+(\xi\alpha)d\eta-4cd(\eta^\perp\phi)=0.
\end{align}

The following lemma is essentially \cite[Lemma 4.2]{berndt-lee-suh} and \cite[Lemma 3.2]{jeong} but with some additional information.

\begin{lemma}
\label{lem:hopf_eigenvalue-1}
Let $M$ be a Hopf hypersurface in $\hat M^m(c)$. If $\xi\alpha=0$, then
\begin{enumerate}
\item[(a)] $A\xi^\perp=\alpha\xi^\perp$,  $A\phi^2\xi^\perp=\alpha\phi^2\xi^\perp$,
\item[(b)] $\alpha A\phi\xi^\perp=(\alpha^2+4c-4c||\xi^\perp||^2)\phi\xi^\perp$.
\end{enumerate}
 \end{lemma}
\begin{proof}
By  (\ref{eqn:d_xi_alpha}), we obtain $\eta^\perp A-\alpha\eta^\perp$; equivalently,  $A\xi^\perp=\alpha\xi^\perp$.
Next, $A\phi^2\xi^\perp=A(||\xi^\perp||^2\xi-\xi^\perp)=\alpha\phi^2\xi^\perp$. 
Finally, we can obtain (b) by using  (a) and (\ref{eqn:AphiA}).
\end{proof}

Now we shall derive some properties of the principal curvatures and their corresponding principal directions for a Hopf hypersurface 
in $\hat M^m(c)$.
Observe that we can derive the following two equations from (\ref{eqn:AphiA}).
\begin{align*}
A\phi A\phi-\frac\alpha 2(\phi A\phi-A)-c(\phi^2+\theta)
=c\sum^3_{a=1}\{-\xi_a\otimes\eta_a+\phi\xi_a\otimes\eta_a\phi\} \\
+c(2\xi\otimes\eta^\perp+2\xi^\perp\otimes\eta)+\left\{\frac{\alpha^2}2-2c||\xi^\perp||^2\right\}\xi\otimes\eta
	\\
\phi A\phi A+\frac\alpha 2(A-\phi A\phi)-c(\phi^2+\theta)
=c\sum^3_{a=1}\{\phi\xi_a\otimes\eta_a\phi-\xi_a\otimes\eta_a\} \\
+c(2\xi\otimes\eta^\perp+2\xi^\perp\otimes\eta)+\left\{\frac{\alpha^2}2-2c||\xi^\perp||^2\right\}\xi\otimes\eta.
\end{align*}
These two equations imply that 
\begin{align}\label{eqn:commute}
A(\phi A\phi)=(\phi A\phi)A. 
\end{align}
Hence 
there exists a local orthonormal frame $\{X_0=\xi,X_1,\cdots,X_{4m-2}\}$ such that $AX_j=\lambda_jX_j$ and 
$\phi A\phi X_j=-\mu_j X_j$ for $j\in\{1,\cdots,4m-2\}$. With this setting,  (\ref{eqn:AphiA})  gives 

\begin{lemma}\label{lem:hopf_eigenvalue}
Let $M$ be a Hopf hypersurface in $\hat M^m(c)$. Then there exists a local orthonormal frame $\{X_0=\xi,X_1,\cdots,X_{4m-2}\}$ 
such that $AX_j=\lambda_jX_j$ and $A\phi X_j=\mu_j\phi  X_j$ for $j\in\{1,\cdots,4m-2\}$. Furthermore, 
for each $x\in M$ with $||\xi^\perp||>0$,
we have
\begin{align}
0=&\left\{\lambda_j\mu_j-\frac{\alpha}2(\lambda_j+\mu_j)-c+c||\xi^\perp||\right\}X_j^+ \label{eqn:hopf_+}\\
0=&\left\{\lambda_j\mu_j-\frac{\alpha}2(\lambda_j+\mu_j)-c-c||\xi^\perp||\right\}X_j^- \label{eqn:hopf_-}\\
0=&\left\{\lambda_j\mu_j-\frac{\alpha}2(\lambda_j+\mu_j)-2c\right\}g(X_j,\xi_a)   		 
			+2c\eta_a(\xi)g(X_j,\xi^\perp)  																						\label{eqn:hopf_xi2}\\
0=&\left\{\lambda_j\mu_j-\frac{\alpha}2(\lambda_j+\mu_j)-2c\right\}g(X_j,\phi\xi_a)
			+2c\eta_a(\xi)g(X_j,\phi\xi^\perp)         																	\label{eqn:hopf_phixi2}\\   
0=&\left\{\lambda_j\mu_j-\frac{\alpha}2(\lambda_j+\mu_j)-2c+2c||\xi^\perp||^2\right\}
			g(X_j,\phi\xi^\perp) 																												\label{eqn:hopf_phixi1}\\
0=&\left\{\lambda_j\mu_j-\frac{\alpha}2(\lambda_j+\mu_j)-2c+2c||\xi^\perp||^2\right\}
			g(X_j,\xi^\perp) 																										\label{eqn:hopf_xi1}
\end{align}
where $X_j^+$ and $X_j^-$ is the component of $X_j$ in $\mathcal H(1)$ and $\mathcal H(-1)$ respectively.
\end{lemma}

\begin{lemma}\label{lem:hopf_eigenvalue-2}
Let $M$ be a Hopf hypersurface in $\hat M^m(c)$. If $0<||\xi^\perp||< 1$, then 
\begin{enumerate}
\item[(a)] $A\mathcal H(1)\subset \mathcal H(1)$,
\item[(b)] $A\mathcal H(-1)\subset\mathcal H(-1)$,
\item[(c)] $A(\mathbb R\phi\xi^\perp\oplus\mathbb R\phi^2\xi^\perp)\subset
\mathbb R\phi\xi^\perp\oplus\mathbb R\phi^2\xi^\perp$.
\end{enumerate}
\end{lemma}
\begin{proof}
Consider the local orthonormal frame stated in Lemma~\ref{lem:hopf_eigenvalue}.

\medskip
(a) Suppose $X_j^+\neq0$ for some $j\in\{1,\cdots,4m-2\}$. Then  (\ref{eqn:hopf_+})--(\ref{eqn:hopf_xi1}) give
\begin{align*}
0=&-2c||\xi^\perp||X_j^- 																														\\
0=&-c(||\xi^\perp||+1)g(X_j,\xi_a) +2c\eta_a(\xi)g(X_j,\xi^\perp)  										\\
0=&-c(||\xi^\perp||+1)g(X_j,\phi\xi_a)+2c\eta_a(\xi)g(X_j,\phi\xi^\perp)       	\\   
0=&c(2||\xi^\perp||+1)(||\xi^\perp||-1)\left\{g(X_j,\phi\xi^\perp)^2+g(X_j,\xi^\perp) 		^2\right\}. 				
\end{align*}
These imply that $X_j\in\mathcal H(1)$ and so we obtain $A\mathcal H(1)\subset\mathcal H(1)$.

\medskip
(b)
Suppose $X_j^-\neq0$ for some $j\in\{1,\cdots,4m-2\}$.
If $||\xi^\perp||\neq 1/2$ at a point $x$, then $X_j\perp$ $\phi \xi^\perp$,  $\xi^\perp$ by (\ref{eqn:hopf_-}), (\ref{eqn:hopf_phixi1})--(\ref{eqn:hopf_xi1}).
Furthermore, since $||\xi^\perp||\neq1$, we obtain $A\mathcal H(-1)\subset \mathcal H(-1)$ at $x$ by (\ref{eqn:hopf_xi2})--(\ref{eqn:hopf_phixi2}).

Now suppose $||\xi^\perp||=1/2$ on an open subset $G\subset M$. 
Then $A\phi\xi^\perp=0$ by virtue of Lemma~\ref{lem:global}(f). It follows further on from (\ref{eqn:AphiA}) that
$A\phi^2\xi^\perp=-(4c/\alpha)(1-||\xi^\perp||^2)\phi^2\xi^\perp$. Hence we can select another orthonormal frame  in which 
$X_{4m-3}=(4/\sqrt3) \phi\xi^\perp,X_{4m-2}=(4/\sqrt3)\phi^2\xi^\perp$. 
It follows that  for $j\in\{1,\cdots,4m-4\}$, 
(\ref{eqn:hopf_-})--(\ref{eqn:hopf_phixi2}) imply that if $X_j^-\neq0$, then $X_j\perp$ $ \xi_a,\phi\xi_a$ on $G$.
Hence, we conclude that $A\mathcal H(-1)\subset \mathcal H(-1)$.

\medskip
(c) 
Since $\mathcal H$ is invariant under $\phi$ and $A$,
we can reconstruct the local orthonormal frame such that 
$X_1,\cdots,X_6$ (resp. $X_7,\cdots,X_{4m-1}$) are tangent to $\mathcal H^\perp$ (resp. $\mathcal H$).
Taking the vectors  $\xi_1,\xi_2,\xi_3$  with properties (\ref{eqn:natural-basis}),   (\ref{eqn:hopf_xi2})--(\ref{eqn:hopf_xi1}) give 
\begin{align*}
0=&\left\{\lambda_j\mu_j-\frac{\alpha}2(\lambda_j+\mu_j)-2c\right\}
					\{g(X_j,\xi_a)^2+g(X_j,\phi\xi_a)^2\}, \quad a\in\{2,3\}								\\   
0=&\left\{\lambda_j\mu_j-\frac{\alpha}2(\lambda_j+\mu_j)-2c+2c\eta_1(\xi)^2\right\}
					\{g(X_j,\phi\xi_1)^2+g(X_j,\xi_1)^2\}. 																									
\end{align*}
These imply that $A(\mathbb R\phi\xi^\perp\oplus\mathbb R\phi^2\xi^\perp)\perp$ $\xi_a$, $\phi\xi_a$
 for $a\in\{2,3\}$ and so the desired result is obtained. 
\end{proof}


\begin{lemma}\label{lem:corvariant}
Let $M$ be a Hopf hypersurface in $\hat M^m(c)$ such that  $0<||\xi^\perp||< 1$ on $M$. Suppose $X\in\mathcal H(\varepsilon)$ such that 
$AX=\lambda X$ and $A\phi X=\mu\phi X$. Then
\begin{align*}
\nabla_X\phi\xi^\perp=\lambda||\xi^\perp||(\varepsilon+||\xi^\perp||)X; \quad 
\nabla_X\phi^2\xi^\perp=\lambda||\xi^\perp||(\varepsilon+||\xi^\perp||)\phi X. \quad 
\end{align*}
Furthermore, if we put 
\begin{align}\label{eqn:ya}
A\phi\xi^\perp=u\phi\xi^\perp-v\phi^2\xi^\perp; \quad  A(-\phi^2\xi^\perp)=p\phi\xi^\perp-q\phi^2\xi^\perp,
\end{align}
then for any $Y\in\mathcal H$
\begin{align*}
g((\nabla_XA)\phi \xi^\perp,Y)=&\lambda||\xi^\perp||(\varepsilon+||\xi^\perp||) \{(u-\lambda)g(X,Y)-vg(\phi X,Y)\}\\
g((\nabla_XA)\phi^2 \xi^\perp,Y)=&\lambda||\xi^\perp||(\varepsilon+||\xi^\perp||) \{-pg(X,Y)+(q-\mu)g(\phi X,Y)\}.
\end{align*}
\end{lemma}
\begin{proof}
Note that   $\theta X=\varepsilon||\xi^\perp||X$ and $\phi^\perp X=-\varepsilon||\xi^\perp||\phi X$. Then by Lemma~\ref{lem:global} we obtain
\begin{align*}
\nabla_X\phi^2\xi^\perp=\nabla_X(-\xi^\perp+||\xi^\perp||^2\xi)
=-\phi^\perp AX+||\xi^\perp||^2\phi AX=\lambda||\xi^\perp||(\varepsilon+||\xi^\perp||)\phi X.
\end{align*}
The other can be obtained similarly. Next under the setting of (\ref{eqn:ya}),  we obtain
\begin{align*}
g((\nabla_XA)\phi \xi^\perp,Y)
=&g(u\nabla_X\phi\xi^\perp-v\nabla_X\phi^2\xi^\perp-A\nabla_X\phi\xi^\perp,Y)\\
=&\lambda||\xi^\perp||(\varepsilon+||\xi^\perp||) \{(u-\lambda)g(X,Y)-vg(\phi X,Y)\}.
\end{align*}
The last identity can be obtained by a similar calculation.
\end{proof}
\section{Mixed foliate  real hypersurfaces in $\hat M^m(c)$}
Let $M$ be a submanifold in a K\"ahler manifold $\hat M$. If the dimension of the maximal holomorphic subspace  
$\mathcal C_x=JT_xM\cap T_xM$, $x\in M$ is constant and its orthogonal complementary distribution $\mathcal C^\perp$ in $TM$ is totally real,
then $M$ is called a \emph{CR-submanifold}.  
If  $\dim \mathcal C\neq0$ and $\dim \mathcal C^\perp\neq0$, then the CR-submanifold $M$ is said to be \emph{proper}.
A CR-submanifold $M$ is said to be \emph{mixed totally geodesic} if  $h(\mathcal C,\mathcal C^\perp)=0$,
where $h$ is the second fundamental form of $M$. 
A \emph{mixed foliate} CR-submanifold $M$ is a mixed totally geodesic CR-submanifold such that  the distribution $\mathcal C$ is integrable
(cf. \cite{chen1}).

A real hypersurface is a typical example of a proper $CR$-submanifold in a K\"ahler manifold with $\mathcal C^\perp=\mathbb R\xi$.
It is clear that  $M$ is mixed totally geodesic if and only if it is Hopf.  
Furthermore by a result in \cite{chen1}, we can state
\begin{lemma}
Let $M$ be a real hypersurface in a K\"ahler manifold. Then $M$ is mixed foliate if and only if $\phi A+A\phi=0$. 
\end{lemma}
In this section, we shall prove the nonexistence of  mixed foliate real hypersurfaces in $\hat M^{m}(c)$.
\begin{theorem}\label{thm:mixed}
There does not exist any mixed foliate real hypersurface  in $\hat M^{m}(c)$, $m\geq3$.
\end{theorem}
\begin{remark}
The nonexistence of mixed foliate real hypersurfaces in a non-flat complex space form was obtained in \cite{ki-suh}.
\end{remark}
The proof of Theorem~\ref{thm:mixed} is splitted into several parts.
We first prove
\begin{lemma} \label{lem:mixed-1}
Let $M$ be a real hypersurface in $\hat M^{m}(c)$.  Then $\phi A+A\phi\neq0$ on each open subset $G\subset M$ with  $\xi^\perp=0$.  
\end{lemma}
\begin{proof}
Since $\xi^\perp=0$ on $G$,    $A\phi \xi_a=0$ for $a\in\{1,2,3\}$ by Lemma~\ref{lem:Aphixi_a}.
Suppose $\phi A+A\phi=0$ on $G$. Then $A\xi_a=\phi A\phi\xi_a=0$ and  $A\xi=\alpha\xi$;  so $G$ is an open part of
 one of the real hypersurfaces 
given in Theorem~\ref{thm:+} and Theorem~\ref{thm:-}.
However, along the direction $\xi_a$, the principal curvature is non-zero for these real hypersurfaces according to Theorem~\ref{thm:+2} and Theorem~\ref{thm:-2}; a contradiction and so the result is obtained.
\end{proof}

We observe that if $\phi A+A\phi=0$, then $M$ is Hopf.  Moreover, 
for each $x\in M$ with  $||\xi^\perp||>0$, by taking a principal curvature vector $X_j$ in 
Lemma~\ref{lem:hopf_eigenvalue} with $X_j^-\neq0$, we obtain $-\lambda_j^2-c(1+||\xi^\perp||)=0$ from $(\ref{eqn:hopf_-})$ and so $c<0$.
It follows that we only need to consider the case $c<0$.
We shall prepare some results before proceeding to the proof of  Theorem~\ref{thm:mixed}.

\begin{lemma}\label{lem:eta-parallel-a}
Let $M$ be a Hopf hypersurface in $\hat M^m(c)$. Suppose $G\subset M$ is an open set with $0<||\xi^\perp||<1$.
If $(\phi A+A\phi)\mathcal H(-1)=0$ on $G$, then for $Y,Z,W\in\mathcal H(-1)$ on $G$, we have
\begin{enumerate}
\item[(a)] $A^2Y=-c(1+||\xi||^\perp)Y$,
\item[(b)] $g((\nabla_YA)Z,W)=0$.
\end{enumerate}
\end{lemma}
\begin{proof}
(a) It can be obtained directly from (\ref{eqn:AphiA}) and the fact $\phi^\perp Y=||\xi^\perp||\phi Y$ for $Y\in\mathcal H(-1)$.

\medskip
(b)  For all vector fields $Y,Z,W$ tangent to $\mathcal H(-1)$ on $G$, it follows from (a),   Lemma~\ref{lem:global}(f) and the Codazzi equation that 
\begin{align*}
0=&g((\nabla_WA)Y,AZ)+g((\nabla_WA)Z,AY)+g(A^2\nabla_WY,Z)+g(A^2Y,\nabla_WZ)\\
     &+c(1+||\xi^\perp||)\{g(\nabla_WY,Z)+g(Y,\nabla_WZ)\}\\
  =&g((\nabla_WA)Y,AZ)+g((\nabla_ZA)W,AY)
	\end{align*}
By taking  a cylic sum over $Y,Z,W$ in the preceding equation, and then substracting the obtained equation from the preceding equation, yields
\[
g((\nabla_YA)Z,AW)=0, \quad Y,Z,W\in\mathcal H(-1).
\] 
By (a) and Lemma~\ref{lem:hopf_eigenvalue-2}(b),  $A$ is an isomorphism when restricted to $\mathcal H(-1)$. Hence, we obtain the lemma. 
\end{proof}

\begin{lemma}\label{lem:eta-parallel-b}
Let $M$ be a Hopf hypersurface in $\hat M^m(c)$. Suppose $G\subset M$ is an open set with $||\xi^\perp||=1$.
If $(\phi A+A\phi)\mathcal W=0$ on $G$, where $\mathcal W=\mathcal H(-1)\oplus(\mathfrak D^\perp\ominus\mathbb R\xi)$, then for $Y,Z,W\in\mathcal W$ on $G$, we have
\begin{enumerate}
\item[(a)] $A^2Y=-2cY$,
\item[(b)] $g((\nabla_YA)Z,W)=0$.
\end{enumerate}
\end{lemma}
\begin{proof}
(a) It can be obtained directly from (\ref{eqn:AphiA}).

\medskip
(b) Using a similar manner as in the proof of  Lemma~\ref{lem:eta-parallel-a} but  with the help of Lemma~\ref{lem:general1},  we obtain 
\[
g((\nabla_YA)Z,AW)=0, \quad Y,Z,W\in\mathcal W.
\] 
We note that $A\mathcal W\subset \mathcal W$ by virtue of Lemma~\ref{lem:theta_A}, together with (a), 
we obtain that  $A_{|\mathcal W}$ is an isomorphism. Hence, we obtain the lemma. 
\end{proof}

\begin{lemma}\label{lem:RA}
Let $M$ be a Hopf  hypersurface $M$ in $\hat M^m(c)$ and $||\xi^\perp||>0$ at a point $x\in M$.
Suppose  $\mathcal V$ is a subspace of $\mathcal H(-1)$ at  $x$ that is  invariant under $A$ and $\phi$.
If  $(\phi A+A\phi)\mathcal V=0$, then 
\begin{align*}
\sum^n_{j=1}g( (R(e_j,\phi e_j)A)Z,W)=-4c(5+||\xi^\perp||+2n)g(\phi AZ,W)
\end{align*} 
for any $Z,W\in\mathcal V$, where $\{e_1,\cdots,e_n\}$ is an orthonormal basis of $\mathcal V$ and $n=\dim \mathcal V$.
\end{lemma}
\begin{proof}
We take $\xi_1,\xi_2,\xi_3$ with properties (\ref{eqn:natural-basis}).
Then under this situation,  $\phi_1 X=\phi X$ and  $\phi_bX$, $\theta_bX\in\mathcal H_1(1)$, $b\in\{2,3\}$, for $X\in\mathcal H(-1)$. 
It follows from the Gauss equation and Lemma~\ref{lem:eta-parallel-a}--\ref{lem:eta-parallel-b} that 
\begin{align*}
g((R(X,Y)A)Z,W)
=&c(3+||\xi^\perp||)\{g(Y,AZ)g(X,W)-g(X,AZ)g(Y,W)\\
  &-g( Y,Z)g(AX,W)+g(X,Z)g(AY,W)\}\\
 &+2c\{g(\phi Y,AZ)g(\phi X,W)-g(\phi X,AZ)g(\phi Y,W)\\
 &    -g(\phi Y,Z)g(A\phi X,W)+g(\phi X,Z)g(A\phi Y,W)\\
& -4g(\phi X,Y)g(\phi AZ,W)\}
\end{align*}
for any $X,Y,Z,W\in\mathcal V$. Hence the lemma can be obtained directly from the preceding equation.
\end{proof}

\begin{lemma}\label{lem:inv-subsp-1}
Let $M$ be a Hopf hypersurface in $\hat M^m(c)$. Suppose $G\subset M$ is an open set with $0<||\xi^\perp||<1$.
If $(\phi A+A\phi)\mathcal H=0$ on $G$, then $\nabla_XZ\perp \mathcal H(1)$  
for all vector fields $X,Z$ tangent to $\mathcal H(-1)$ on $G$.  
\end{lemma}
\begin{proof}
By the hypothesis $(\phi A+A\phi)\mathcal H=0$ and Lemma~\ref{lem:hopf_eigenvalue}--\ref{lem:hopf_eigenvalue-2}, we can select local orthonormal principal vector fields 
$,X_1, X_2,\cdots,X_{4m-8}$ 
such that  
$X_j$,  $X_{m-2+j}=\phi X_j$  are tangent to $\mathcal H(-1)$
with $\lambda_j=\lambda=-\lambda_{m-2+j}$, where $\lambda=\sqrt{-c(1+||\xi^\perp||)}$, 
and $X_{2m-4+j},\cdots, X_{3m-6+j}=\phi X_{2m-4+j}$ are tangent to $\mathcal H(1)$, $j\in\{1,\cdots,m-2\}$. 
We can further deduce from  (\ref{eqn:hopf_+})--(\ref{eqn:hopf_-})  that $\lambda_r\neq\lambda$ for $r\in\{2m-3,\cdots,4m-8\}$.

Fixed $i,j\in\{1,\cdots, m-2\}$ and $r\in\{2m-3,\cdots,4m-8\}$, 
It follows from  the Codazzi equation and  Lemma~\ref{lem:global}(f)  that  
\[
g(\nabla_{X_i}X_j,X_r)=-\frac1{(\lambda_r-\lambda)}g(\nabla_{X_i}A)X_r-(\nabla_{X_r}A){X_i},{X_j})=0.
\]
We can further deduce from the preceding equation that 
\[
g(\nabla_{X_i}\phi X_j,X_r)=g((\nabla_{X_i}\phi)X_j,X_r)-g(\nabla_{X_i}X_j,\phi X_r)=0.
\]
Similarly, we have $g(\nabla_{\phi X_i}\phi X_j,X_r)=g(\nabla_{\phi X_i}X_j,X_r)=0$. 
This completes the proof.
\end{proof}

For a real hypersurface $M$ in $\hat M^m(c)$, if  $G\subset M$ is an open set with $||\xi^\perp||=1$,
then $A\mathcal H(1)=0$ (and so $(\phi A+A\phi)\mathcal H(1)=0$) on $G$ according to Lemma~\ref{lem:theta_A}.
Based on this observation, although there is a slight difference between the hypotheses, the following lemma can be obtained in a similar manner as in the proof  of Lemma~\ref{lem:inv-subsp-1}.
\begin{lemma}\label{lem:inv-subsp-2}
Let $M$ be a Hopf hypersurface in $\hat M^m(c)$. Suppose $G\subset M$ is an open set with $||\xi^\perp||=1$
and $\mathcal V\subset\mathcal H(-1)$ is a subbundle  over $G$ that is invariant under $A$ and $\phi$.
If $(\phi A+A\phi)\mathcal V=0$ on $G$, then $\nabla_XZ\perp \mathcal H(1)$  
for all vector fields $X,Z$ tangent to $\mathcal V$ on $G$.

\end{lemma}

\begin{lemma}\label{lem:mixed-2}
Let $M$ be a Hopf  hypersurface in $\hat M^{m}(c)$, $m\geq3$.  Then $(\phi A+A\phi)\mathcal H\neq0$ on each open subset $G\subset M$ with $0<||\xi^\perp||<1$.
\end{lemma}
\begin{proof}
Suppose $(\phi A+A\phi)\mathcal H=0$ on $G$. 
Then by Lemma~\ref{lem:eta-parallel-a}, we obtain 
\begin{align} \label{eqn:RA-1}
&g((R(X,Y)A)Z,W)+g((\nabla_YA)\nabla_XZ,W)+g((\nabla_YA)Z,\nabla_XW)		\nonumber\\
&+g((\nabla_{[X,Y]}A)Z,W)-g((\nabla_XA)\nabla_YZ,W)-g((\nabla_XA)Z,\nabla_YW)=0
\end{align}
for any vector fields $X,Y,Z,W$ tangent to $\mathcal H(-1)$, here  we have used the fact
\begin{align*}
(R(X,Y)A)Z=&\nabla^2A(;Y;X)Z-\nabla^2A(;X;Y)Z
\end{align*}
where 
\[
\nabla^2A(;Y;X)Z:=\nabla_X(\nabla_YA)Z-(\nabla_{\nabla_XY}A)Z-(\nabla_YA)\nabla_XZ. 
\]
By using  
Lemma~\ref{lem:global}(h)--(j),  
Lemma~\ref{lem:hopf_perp-1a},  
Lemma~\ref{lem:eta-parallel-a} and Lemma~\ref{lem:inv-subsp-1}, on one hand, we obtain
\begin{align} \label{eqn:[]}
g((\nabla_{[X,Y]}A)Z,W)=0
\end{align}
and on the other hand
\begin{align} \label{eqn:derivative}
g((\nabla_YA)  &  \nabla_XZ,W)   \nonumber\\
=&\eta(\nabla_XZ)g((\nabla_YA)\xi,W)
			+\frac{g(\nabla_XZ,\phi^2\xi^\perp)}{||\xi^\perp||^2(1-||\xi^\perp||^2)}g((\nabla_YA)\phi^2\xi^\perp,W)   \nonumber\\
	&+\frac{g(\nabla_XZ,\phi\xi^\perp)}{||\xi^\perp||^2(1-||\xi^\perp||^2)}g((\nabla_YA)\phi\xi^\perp,W) \nonumber \\
=&-g(\phi AX,Z)g((\nabla_YA)\xi,W)
			+\frac{g(\phi AX,Z)}{||\xi^\perp||(1+||\xi^\perp||)}g((\nabla_YA)\phi^2\xi^\perp,W)   \nonumber\\
	&+\frac{g(AX,Z)}{||\xi^\perp||(1+||\xi^\perp||)}g((\nabla_YA)\phi\xi^\perp,W) 
\end{align}
for any vector fields $X,Y,Z,W$ tangent to $\mathcal H(-1)$.

Let $\{e_1,\cdots,e_{2m-4}\}$ be an orthonormal basis of $\mathcal H(-1)$ and 
$Z$ be a unit vector field tangent to $\mathcal H(-1)$ such that $AZ=\lambda Z$ (and so $A\phi Z=-\lambda\phi Z$),  where $\lambda=\sqrt{-c(1+||\xi^\perp||)}$.
Then by using (\ref{eqn:RA-1})--(\ref{eqn:derivative}) and Lemma~\ref{lem:corvariant}, we obtain
\begin{align*} 
\sum^{2m-4}_{j=1}((   &  R(e_j,\phi e_j)   A)Z,\phi Z)   \nonumber\\
=&-2\lambda g((\nabla_{Z}A)\phi Z- (\nabla_{\phi Z}A)Z,\xi)		
  +\frac{2\lambda g((\nabla_{Z}A)\phi Z-(\nabla_{\phi Z}A)Z,\phi^2\xi^\perp)	}{||\xi^\perp||(1+||\xi^\perp||)} \nonumber\\
  & -2\lambda\frac{ g((\nabla_{\phi Z}A)\phi\xi^\perp,\phi Z)
					+g(\nabla_{Z}A)\phi\xi^\perp, Z)	}{||\xi^\perp||(1+||\xi^\perp||)}			\\
=&-2\lambda\{-2c(1+||\xi^\perp||)\}+2\lambda\{2c(1-||\xi^\perp||)\}-2\lambda\{-2c(1-||\xi^\perp||)\}		\\	
 =&4c\lambda(3-||\xi^\perp||).	
\end{align*}
This, together with Lemma~\ref{lem:RA}, gives $16c\lambda m=0$.
This is a contradiction  and so the proof is completed.
\end{proof}

\begin{lemma}\label{lem:mixed-3}
Let $M$ be a real hypersurface in $\hat M^{m}(c)$, $m\geq 3$.  Then $\phi A+A\phi\neq0$ on each open subset $G\subset M$.
\end{lemma}
\begin{proof}
Suppose $\phi A+A\phi=0$ on an open subset $G\subset M$. 
By virtue of Lemma~\ref{lem:mixed-1} and Lemma~\ref{lem:mixed-2},  $||\xi^\perp||=1$ on $G$ or $\xi=\xi^\perp\in\mathfrak D^\perp$.
We consider the vectors $\xi_1,\xi_2,\xi_3$ with properties (\ref{eqn:natural-basis})--(\ref{eqn:natural-basis2}).

We first prove that $A\mathfrak D^\perp\subset \mathfrak D^\perp$. It suffices to show that $A\xi_2$, $A\xi_3\in\mathfrak D^\perp$.
Suppose $A\xi_2=p\xi_2+q\xi_3+rU$, where $U$ is a unit vector field tangent to $\mathcal H(-1)$ and $r$ is a non-vanishing function on $G$.
Then by the hypothesis $\phi A+A\phi =0$ and Lemma~\ref{lem:eta-parallel-b}, we can obtain
\begin{align*}
A\xi_3	=&q\xi_2-p\xi_3+r\phi U\\
AU				=&r\xi_2-pU-q\phi U\\
A\phi U=&r\xi_3-qU+p\phi U.
\end{align*}
It follows that $\mathcal V:=\mathcal H(-1)\ominus\vspan\{U,\phi U\}$ is invariant under $A$ and $\phi$.
By Lemma~\ref{lem:eta-parallel-b}, we obtain
\begin{align*}
&g((R(X,Y)A)Z,W)+g((\nabla_YA)\nabla_XZ,W)+g((\nabla_YA)Z,\nabla_XW)		\nonumber\\
&+g((\nabla_{[X,Y]}A)Z,W)-g((\nabla_XA)\nabla_YZ,W)-g((\nabla_XA)Z,\nabla_YW)=0
\end{align*}
for any vector fields $X,Y,Z,W$ tangent to $\mathcal V$. 
On the other hand,  by Lemma~\ref{lem:eta-parallel-b} and Lemma~\ref{lem:inv-subsp-2}, we obtain
\begin{align*} 
g((\nabla_YA)\nabla_XZ,W) =&\eta(\nabla_XZ)g((\nabla_YA)\xi,W)		\nonumber\\
=&-g(\phi AX,Z)g(\alpha\phi AY-2c\phi Y,W)
\end{align*}
for any vector fields $X,Y,Z,W$ tangent to $\mathcal V$.
By using these two equations, we obtain
\begin{align*}
\sum^{2m-4}_{j=1}((R(e_j,\phi e_j)A)Z,W)-8cg(\phi AZ,W)=0, \quad Z,W\in\mathcal V
\end{align*}
where $\{e_1,\cdots,e_{2m-4}\}$ is an orthonormal basis for $\mathcal V$.
This, together with Lemma~\ref{lem:RA}, gives
\[
-16cmg(\phi AZ,W)=0, \quad Z,W\in\mathcal V.
\]
This contradicts the fact that $A$ is an isomorphism on $\mathcal V$. Hence $A\xi_2\in\mathfrak D^\perp$ and so
$A\xi_3=\phi A\xi_2\in\mathfrak D^\perp$. Accordingly,  $A\mathfrak D^\perp\subset \mathfrak D^\perp$.
It follows that  $G$ is an open part of one of the real hypersurfaces 
stated  in Theorem~\ref{thm:+} and Theorem~\ref{thm:-}.
However, the fact $(\phi A+A\phi)\xi_a=0$ prevents $M$ from being any one of the cases in these two theorems in light of  Theorem~\ref{thm:+2} and Theorem~\ref{thm:-2}; it is a contradiction and so the proof is completed.
\end{proof}

\begin{proof}[Proof of Theorem~\ref{thm:mixed}] It is an immediate consequence of Lemma~\ref{lem:mixed-3}.
\end{proof}

By using  Theorem~\ref{thm:mixed}, we obtain the following general properties of Hopf hypersurfaces.
\begin{theorem}\label{thm:hopf_eigenvalue-3}
Let $M$ be a Hopf hypersurface in $\hat M^m(c)$, $m\geq3$.  Then  
\begin{enumerate}
\item[(a)] $\xi\alpha=0$;   $d\alpha=-4c\eta^\perp\phi$,
\item[(b)]  $\alpha$ is constant if and only if either $\xi\in\mathfrak D$ or $\xi\in\mathfrak D^\perp$,
\item[(c)] $A\xi^\perp=\alpha\xi^\perp$,  $A\phi^2\xi^\perp=\alpha\phi^2\xi^\perp$,
\item[(d)] $\alpha A\phi\xi^\perp=(\alpha^2+4c-4c||\xi^\perp||^2)\phi\xi^\perp$.
\end{enumerate}
\end{theorem}
\begin{proof}
(a)  In  each open subset $G\subset M$ with $||\xi^\perp||=0$,  $\theta=\phi^\perp=0$ and so 
$\xi\in\mathfrak D$. 
Then one of the cases in Theorem~\ref{thm:lee-suh} occur,  and so $\alpha$ is constant on $G$; this gives $\xi\alpha=0$ on $G$.
Next,  for each $x\in M$ with  $0<||\xi^\perp||<1$,  
it follows from  (\ref{eqn:201}),  (\ref{eqn:210})  
and   Lemma~\ref{lem:hopf_eigenvalue-2} that 
\[
(\xi\alpha)g((\phi A+A\phi)X,Y)=(\xi\alpha)d\eta(X,Y)=0, \quad X,Y\in\mathcal H.
\]
Hence, we  obtain $\xi\alpha=0$ at $x$ by Lemma~\ref{lem:mixed-2}.
Now consider an open subset $G\subset M$ with  $||\xi^\perp||=1$.  Then $\eta=\eta^\perp$ and so
(\ref{eqn:210}) descends to  $(\xi\alpha)d\eta=0$.
It follows from Lemma~\ref{lem:mixed-3} that $\xi\alpha=0$ on $G$. 
Consequently, we conclude that  $\xi\alpha=0$ on $M$. Next by (\ref{eqn:d_alpha}), we obtain $d\alpha=-4c\eta^\perp\phi$.

\medskip
(b)--(d) These can be obtained immediately  from (a) and Lemma~\ref{lem:hopf_eigenvalue-1}.

\end{proof}

With the help of Theorem~\ref{thm:hopf_eigenvalue-3}, we can  complete the classfication problem of contact real hypersurfaces in $\hat M^m(c)$, $c<0$, considered by Berndt, Lee and Suh  in \cite{berndt-lee-suh}. 

Recall that a real hypersurface $M$ in a K\"ahler manifold is said to be \emph{contact} if  $\phi A+A\phi=\rho\phi$ for a nowhere zero function $\rho$ on $M$. This means that the almost contact metric structure $(\phi,\xi,\eta,g)$ of $M$ is contact up to a $\mathcal C$-homothetic deformation. 
By using Theorem~\ref{thm:hopf_eigenvalue-3} and \cite[Theorem 1.1]{berndt-lee-suh}, we obtain the following result.

\begin{theorem}\label{thm:contact}
Let $M$ be a real hypersurface in $SU_{2,m}/S(U_2U_m)$, $m\geq 3$. Then $M$ is contact if and only if it is an open part of one of  real hypersurfaces of type $B$ or $C_2$ given  in Theorem~\ref{thm:-}.
\end{theorem}

\begin{remark}
 Theorem~\ref{thm:contact} was obtained in \cite{suh3} for the case $c>0$.
\end{remark}

\section{$q$-umbilical real hypersurfaces in $\hat M^m(c)$}
Recall that a real hypersurfaces $M$ in $\hat M^m(c)$ is said to be \emph{$q$-umbilical} if it satisfies
\begin{align*}
A= f_1\mathbb I+f_2\theta +f_3\sum^3_{a=1}\xi_a\otimes\eta_a
\end{align*}
where $f_1$, $f_2$, $f_3$  are functions on $M$.
This class of real hypersurfaces is interesting to be studied as it includes three important types of real hypersurfaces.
We can easily obtain the following from Theorem~\ref{thm:+2} and Theorem~\ref{thm:-2}.

\begin{lemma}\label{lem:umbilic-eq}
\begin{enumerate}\item[]
\item[(i)] Real hypersurfaces of type $A$ in $SU_{m+2}/S(U_2U_m)$ are $q$-umbilical with 
	\[
	f_1=-f_2=-\frac{\sqrt{2c}\tan(\sqrt{2c}r)}2, \quad f_3=\sqrt{2c}\cot(\sqrt{2c}r), \quad  0<r<\frac{\pi}{\sqrt{8c}}, \quad c>0.
	\]
\item[(ii)] Real hypersurfaces of type $A$ in $SU_{2,m}/S(U_2U_m)$ are $q$-umbilical with 
	\[
	f_1=-f_2=\frac{\sqrt{-2c}\tanh(\sqrt{-2c}r)}2, \quad f_3=\sqrt{-2c}\coth(\sqrt{-2c}r), \quad r>0, \quad c<0.
	\]
	\item[(iii)] Real hypersurfaces of type $C_1$ in $SU_{2,m}/S(U_2U_m)$ are $q$-umbilical with 
	\[
	f_1=-f_2=\frac{\sqrt{-2c}}2, \quad f_3=\sqrt{-2c}, \quad c<0.
	\]
\end{enumerate}
\end{lemma}

We shall consider a more general condition than $q$-umbilicity to classify $q$-umbilical real hypersurfaces as well as to obtain a nonexistence result.
\begin{theorem}\label{thm:umbilic}
Let $M$ be a connected real hypersurface in $\hat M^m(c)$, $m\geq 3$. Suppose $M$ satisfies 
\begin{align}\label{eqn:umbilic}
A= f_1\mathbb I+f_2\theta +f_3\sum^3_{a=1}\xi_a\otimes\eta_a+f_4\xi\otimes\eta
\end{align}
where $f_1$, $f_2$, $f_3$, $f_4$  are functions on $M$. Then $f_4=0$, that is, $M$ is $q$-umbilical. Furthermore, one of the following holds: 
\begin{enumerate}
\item[(i)]  for $c>0$: $M$ is an open part of a real hypersurface of type $A$ given in Theorem~\ref{thm:+};  or
\item[(ii)] for $c<0$: $M$ is an open part of one of real hypersurfaces of type $A$ or $C_1$ given in Theorem~\ref{thm:-}.
\end{enumerate}
\end{theorem}

\begin{proof}
For each $x\in M$ with $||\xi^\perp||>0$, by (\ref{eqn:umbilic}), we obtain
\begin{align} \label{eqn:umbilic2}
\left.\begin{aligned}
AX=&(f_1+\varepsilon f_2||\xi^\perp||)X , \quad X\in\mathcal H(\varepsilon),~ \varepsilon\in\{1,2\}\\
A\xi=&(f_1+f_4)\xi+(f_3-f_2)\xi^\perp\\
A\xi^\perp=&(f_4-f_2)||\xi^\perp||^2\xi+(f_1+f_3)\xi^\perp\\
A\phi\xi^\perp=&(f_1+f_2||\xi^\perp||^2)\phi\xi^\perp.
\end{aligned}\right\}
\end{align}
Consider an open subset $G\subset M$ with  $0<||\xi^\perp||<1$.
Write $\lambda_\varepsilon=f_1+\varepsilon f_2||\xi^\perp||$ and  let $X$, $Y\in\mathcal H(\varepsilon)$. 
Then we have 
\begin{align}\label{eqn:8.3}
\nabla_X\xi=\lambda_\varepsilon \phi X, \quad 
\nabla_X\xi^\perp=\phi^\perp AX=-\varepsilon||\xi^\perp||\lambda_\varepsilon\phi X,
\end{align}
where we have used the fact $\phi^\perp X=-\varepsilon||\xi^\perp||\phi X$.
Hence, by (\ref{eqn:umbilic2})--(\ref{eqn:8.3}), we obtain
\begin{align*}
g((\nabla_XA)\xi,Y)
=&\lambda_\varepsilon (f_4-\varepsilon f_3||\xi^\perp||)g(\phi X,Y)
\end{align*}
It follows from the preceding equation and the Codazzi equation  that 
\begin{align*}
2\lambda_\varepsilon (f_4-\varepsilon f_3||\xi^\perp||)g(\phi X,Y)
		=&g((\nabla_XA)Y-(\nabla_YA)X,\xi) \nonumber\\
		=&-2c(1-\varepsilon||\xi^\perp||)g(\phi X,Y),	
\end{align*}
which gives 
\[
f_1f_4-f_2f_3||\xi^\perp||^2+c=\varepsilon||\xi^\perp||\{f_1f_3-f_2f_4+c\}, \quad \varepsilon\in\{1,-1\}
\]
and so
\begin{align}\label{eqn:402}
f_2f_3||\xi^\perp||^2-f_1f_4=c=f_2f_4-f_1f_3.
\end{align}
Similarly, we  compute
\begin{align*}
2\lambda_\varepsilon ||\xi^\perp||(f_4||\xi^\perp||-\varepsilon f_3)g(\phi X,Y)
=&g((\nabla_XA)\xi^\perp,Y)-g((\nabla_YA)\xi^\perp,X)\\
=&-2c||\xi^\perp||(||\xi^\perp||-\varepsilon)g(\phi X,Y)	
\end{align*}
to obtain
\begin{align}\label{eqn:412}
f_2f_4||\xi^\perp||^2-f_1f_3=c=f_2f_3-f_1f_4.
\end{align}
We can deduce from (\ref{eqn:402})--(\ref{eqn:412}) that 
\begin{align*}
f_2=0, \quad f_3=f_4, \quad f_1f_3=-c.
\end{align*}
It follows that  $f_1+f_2||\xi^\perp||^2=f_1\neq0$.
For $b\in\{1,2,3\}$, by using  Lemma~\ref{lem:f},  (\ref{eqn:201})  and (\ref{eqn:umbilic2}), we obtain
\begin{align*}
0=&(\theta A-A\theta)\xi_b-2\sum^3_{a=1}\{g(A\xi_b,\phi\xi_a)\phi\xi_a-g(\phi\xi_a,\xi_b)A\phi\xi_a\}  \\
=&f_3\Big\{\big(\eta_{b+2}(\xi)^2+\eta_{b+1}(\xi)^2\big)\xi_b -\eta_b(\xi)\eta_{b+1}(\xi)\xi_{b+1}
				-\eta_b(\xi)\eta_{b+2}(\xi)\xi_{b+2}\\
&-\eta_{b+2}(\xi)\phi\xi_{b+1}+\eta_{b+1}(\xi)\phi\xi_{b+2}\Big\}.
\end{align*}
Since $f_3\neq0$ and $\{\xi_1,\xi_2,\xi_3,\phi\xi_1,\phi\xi_2,\phi\xi_3\}$ is linearly independent, 
we have  
$\eta_b(\xi)=0$ for $b\in\{1,2,3\}$. Hence $G$ must be empty and so either $||\xi^\perp||=0$ or $||\xi^\perp||=1$ everywhere.

We first consider $\xi^\perp=0$.  Then $\theta=0$ and $\eta_a(\xi)=0$, $a\in\{1,2,3\}$ in this case. 
Hence by Lemma~\ref{lem:Aphixi_a}, we obtain $A\phi\xi_a=0$ and so $f_1=0$ by virtue of (\ref{eqn:umbilic}).
We can then obtain $A\mathcal H=0$,  $A\xi=f_4\xi$ and $A\xi_a=f_3\xi_a$ for $a\in\{1,2,3\}$ further on.
However, by Theorem~\ref{thm:+}--\ref{thm:-2}, we see that such a real hypersurface does not exist.
Consequently, this case cannot occur.

Finally, suppose that $||\xi^\perp||=1$, which means that $\xi=\xi^\perp\in\mathfrak D^\perp$.
This, together with (\ref{eqn:umbilic}), gives  $A\mathfrak D^\perp\subset\mathfrak D^\perp$.
Then by Theorem~\ref{thm:A_A0}, $M$ is an open part of one of the spaces listed in the theorem.
Furthermore it follows from Lemma~\ref{lem:umbilic-eq} that  $f_4=0$.  
\end{proof}

Recall that a real hypersurface $M$ in a K\"ahler manifold is said to be \emph{$\eta$-umbilical} if  it satisfies
\[
A=u\mathbb I+v\xi\otimes\eta
\]
for some functions $u,v$ on $M$.
By Theorem~\ref{thm:umbilic}, we immediately obtain
\begin{corollary}
There does not exist any $\eta$-umbilical real hypersurface $M$ in $\hat M^{m}(c)$, $m\geq 3$.
\end{corollary}


\end{document}